\documentclass[11pt]{article}
\usepackage[utf8]{inputenc}
\usepackage{graphicx}
\usepackage{amsmath}
\usepackage{color}
\usepackage{appendix}
\usepackage{blkarray}
\usepackage{fullpage}
\usepackage{mathtools}
\usepackage{commath}
\usepackage{csquotes}
\usepackage{amsfonts}
\usepackage{amssymb}
\usepackage{amsthm}
\newtheorem{thm}{Theorem}[section]

\newtheorem{dfn}{Definition}[section]

\newtheorem{rmk}{Remark}[section]

\DeclareMathOperator{\diam}{diam}
\DeclareMathOperator{\diag}{diag}

\DeclareMathOperator{\Tr}{Tr}

\title{A note on the distance and distance signless Laplacian spectral radius of complements of trees}
\author{Iswar Mahato \thanks{Department of Mathematics, Indian Institute of Technology Kharagpur, Kharagpur 721302, India. Email: iswarmahato02@gmail.com, iswarmahato02@iitkgp.ac.in}\  \and M. Rajesh Kannan\thanks{Department of Mathematics, Indian Institute of Technology Hyderabad, Hyderabad 502285, India. Email: rajeshkannan1.m@gmail.com, rajeshkannan@math.iith.ac.in }}
\date{}
\begin{document}
\maketitle

\begin{abstract}
 In this article, we show that the generalized tree shift operation increases the distance spectral radius, distance signless Laplacian spectral radius, and the $D_\alpha$-spectral radius of complements of trees. As a consequence of this result, we correct an ambiguity in the proofs of some of the known results. 

\end{abstract}



\section{Introduction}\label{sec1}
Let $G=(V(G), E(G))$ be a finite, simple, and connected graph with the vertex set $V(G)=\{v_1,v_2,\hdots,v_n\}$ and the edge set $E(G)=\{e_1,e_2,\hdots,e_m\}$. The cardinality of the vertex set $V(G)$ is the \textit{order} of the graph $G$. Two vertices $v_i$ and $v_j$ of $ G$ are \textit{adjacent} if an edge exists between them in $G$. Let $N(v)$ denote the collection of vertices that are adjacent to the vertex $v$ in $G$, and $N(v)$ is called the \textit{neighbour} of $v $ in $G$. The \textit{complement} $\overline{G}$ of $G$ is the graph with $V(\overline{G}) = V(G)$ and two distinct vertices in $\overline{G}$ are adjacent if and only if they are non-adjacent in $G$. The \textit{distance} between the vertices $v_i$ and $v_j$ in $G$, denoted by $d_G(v_i,v_j)$, is the length of a shortest path between them in $G$, and define $d_G(v_i,v_i) =0$  for all $v_i \in V(G)$. The \textit{diameter} of $G$, denoted by  $\diam(G)$, is the greatest distance between any pair of vertices in $G$. The \textit{distance matrix} $D(G)$ of $G$ is the $n \times n$ matrix with its rows and columns are indexed by the vertex set of $G$, and the $(v_i,v_j)$-th entry is equal to $d_G(v_i,v_j)$. The \textit{transmission} of a vertex $v_i$, denoted by $\Tr(v_i)$, is the sum of the distances from $v_i$ to all other vertices of $G$. Let $\Tr(G)=\diag \big(\Tr(v_1),\Tr(v_2),\hdots,\Tr(v_n)\big)$ be the diagonal matrix of vertex transmissions of the vertices of $G$. The \textit{distance signless Laplacian matrix} of $G$ is defined as $D^Q(G)=\Tr(G)+D(G)$. The largest eigenvalue of $D(G)$ and $D^Q(G)$ are called \textit{the distance spectral radius} and the \textit{distance signless Laplacian spectral radius} of $G$, respectively. We denote the distance spectral radius and the distance signless Laplacian spectral radius of $G$ by $\lambda_1(G)$ and $\mu_1(G)$, respectively. The $D_{\alpha}$-matrix of $G$ is defined as 
\begin{align*}
	D_\alpha (G) = \alpha Tr(G) + (1-\alpha) D(G) \quad \text{for any real $\alpha \in [0,1]$}.     
\end{align*}
The largest eigenvalue of $D_\alpha (G)$ is called \textit{the $D_\alpha$-spectral radius} of $G$ and is denoted by $\rho_\alpha(G)$. Let $P_n$ and $K_{1,n-1}$ denote the path and the star on $n$ vertices, respectively.

Let $u, v $ be two vertices in a graph $G$. Let  $G^\prime$ be the graph obtained from $G$ by removing the edges between $v$ and $N(v) \setminus (N(u)\cup \{u\})$ and adding the edges between $u$ and $N(v) \setminus (N(u) \cup \{u\})$. The above operation is known as the Kelmans transformation.  
\begin{rmk}\label{remark}{\rm
For a tree $T$, if $T^\prime$ is obtained from $T$ by applying the Kelmans transformation exactly once, then the number of pendant vertices in $T^{\prime}$ equals the number of pendant vertices in $T$ plus one. The diameter of $T^\prime$ is either equal to $\diam(T)$ or $\diam(T)-1$.  }
\end{rmk}
In \cite{Lin-Dist-comple}, Lin and Drury proved that the Kelmans transformation, under suitable assumptions, reduces the distance spectral radius of complements of trees.

\begin{thm}[{\cite[Theorem $2$]{Lin-Dist-comple}}]\label{transfom 1}
	Let $T$ be a tree on $n$ vertices with $\diam(T)\geq 4$. Let $X=(x_1,x_2,\hdots,x_n)^t$ be the unit Perron vector of $D(T^c)$. Let $uv\in E(T)$ with $x_u\geq x_v$, and 
	$$T^{\prime}=T-\{vw: w\in N_T(v)\setminus \{u\}\}+\{uw: w\in N_T(v)\setminus \{u\}\}.$$
	Then $\lambda_1(\overline{T^{\prime}})\geq \lambda_1(\overline{T})$ with equality if and only if $N_T(v)=\{u\}$.
\end{thm}

\begin{thm}[{\cite[Theorem $3$]{Lin-Dist-comple}}]\label{transfom 2}
	Let $T$ be a tree on $n$ vertices with $\diam(T)\geq 4$. Suppose that $e$ is a non-pendant edge of $T$. Let $T^{\prime}$ be the tree obtained from $T$ by collapsing $e$, identifying the vertices of $e$ to a single vertex $z$, and adding a pendant edge at $z$. Then $\lambda_1(\overline{T^{\prime}})> \lambda_1(\overline{T})$.
\end{thm}

Using these results, Lin and Drury \cite{Lin-Dist-comple} characterized the trees whose complements have the minimum distance spectral radius.

\begin{thm}[{\cite[Theorem $5$]{Lin-Dist-comple}}]\label{Dist-comple-min}
	Let $T$ be a tree on $n\geq 4$ vertices with $T\ncong K_{1,n-1}$. Then $\lambda_1(\overline{T^{\prime}})> \lambda_1(\overline{P_n})$ with equality if and only if $T$ is isomorphic to $P_n$.
\end{thm}

In the proof of Theorem \ref{Dist-comple-min}, Lin and Drury mentioned that any tree $T$ of order $n$ could be obtained from $P_n$ by repeatedly applying the Kelmans transformation as in Theorem \ref{transfom 2}. But this is not true in general. Next, we give some counterexamples.

The tree $\Tilde{T}$, given in Figure \ref{fig1}, can not be obtained from $P_7$ by applying the graft transformation as in Theorem \ref{transfom 2}. Because $\diam(\Tilde{T})=\diam(P_7)-2$ and number of pendant vertices in $\Tilde{T}$ equals the number of pendant vertices in $P_7+1$.
This example is given in \cite{csikvari2010poset}.

\begin{figure}[h!]
	\centering
	\includegraphics[scale= 0.70]{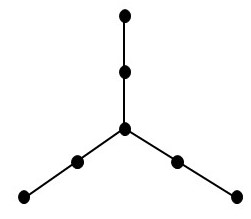}
	\caption{The tree $\Tilde{T}$ on $7$ vertices.}
	\label{fig1}
\end{figure}

Any tree $T$ of order $n$ with $3$ pendant vertices and $\diam(T)\leq n-3$ can not be obtained from $P_n$ by applying the graft transformation as in Theorem \ref{transfom 2}.

So, the proof of Theorem $5$ in \cite{Lin-Dist-comple} has some ambiguity. The same technique is used in the proofs of  Theorem $4.2$ by Li et al. \cite{Li-Distsign-comple} and Theorem $3.9$ by Qin et al. \cite{qin2021dalpha} for characterizing the complements of trees with minimum distance signless Laplacian spectral radius and $D_\alpha$-spectral radius, respectively. 

In this article, we show that the generalized tree shift transformation increases the distance spectral radius, distance signless-Laplacian spectral radius, and the $D_\alpha$-spectral radius of complements of trees. Using these results, we correct the proofs of Theorem $5$ in \cite{Lin-Dist-comple}, Theorem $4.2$ of \cite{Li-Distsign-comple} and Theorem $3.9$ of \cite{qin2021dalpha}.

\section{Main results}

In this section, first, we recall the definition of generalized tree shift and some of its properties. We show that the generalized tree shift increases the distance spectral radius, distance signless-Laplacian spectral radius, and the $D_\alpha$-spectral radius of complements of trees. Using these results, we complete the proofs for characterizing the trees whose complements have minimum distance spectral radius, distance signless Laplacian spectral radius, and  $D_\alpha$-spectral radius, respectively. Next we recall generalized tree shift (GTS) transformation introduced by Csikv\'{a}ri \cite{csikvari2010poset}.

\begin{dfn}[{\cite[Definition 2.1]{csikvari2010poset}}]\label{GTS}
	Let $T$ be a tree and $u,v\in V(T)$ such that all the interior vertices of the path $uPv$ (if they exist) have degree $2$. Let $w$ be the neighbor of $v$ lying on the $uPv$ path (it is possible that $u=w$). Now, construct a tree $T_1$ from $T$ by deleting all the edges between $v$ and $N_T(v)\setminus \{w\}$ and add the edges between $u$ and $N_T(v)\setminus \{w\}$. 
\end{dfn}

  Note that if $u$ or $v$ is a pendant vertex, then $T_1$ isomorphic to  $T$, otherwise the number of pendant vertices in $T_1$ equals the number of pendant vertices in $T$ plus one. In the latter case, the GTS is called proper. Define $T_1>T$ if $T_1$ can be obtained from $T$ by some proper GTS. The relation $>$ induces a poset on the tree on $n$ vertices \cite{csikvari2010poset}.

In \cite{csikvari2010poset}, Csikv\'ari proved the following theorem about the minimal element of the poset.

\begin{thm}[{\cite[Theorem 2.4]{csikvari2010poset}}]\label{GTS-minimal}
	Every tree different from the path is the image of some
	proper generalized tree shift.  
\end{thm}
The star is the unique maximal element, and
the path is the unique minimal element of the induced
poset of the generalized tree shift. Thus, it is easy to see that any tree on $n$ vertices other than the path can be obtained from the path graph $P_n$ by repeatedly applying the generalized tree shift.

In the following theorems, we prove that GTS increases the distance spectral radius, distance signless Laplacian spectral radius, and $D_\alpha$-spectral radius of complements of trees. 

\begin{thm}\label{dist-comple}
	Let $T$ be a tree of order $n$ with $\diam(T)\geq 4$, and let $T_1$ be the tree obtained from $T$ by a proper GTS. Then $\lambda_1(\overline{T_1})> \lambda_1(\overline{T})$.
\end{thm}
\begin{proof}
	Let $X=(x_1,x_2,\hdots,x_n)^t$ be the unit perron vector of $\overline{T}$ corresponding to $\lambda_1(\overline{T})$. Without loss of generality, assume that $x_u\geq x_v$. Therefore, 
	\begin{align*}
		T_1=T-\{vz: z\in S\}+\{uz: z\in S\}, \quad \text{where, $S=N_T(v)\setminus \{w\}$}.    
	\end{align*}
	
	It is easy to see that $D(\overline{T})=2A(T)+J-I$ and $D(\overline{T_1})\geq 2A(T)+J-I$ entrywise. Therefore, by Rayley's principle, we have 
	\begin{align*}
		\lambda_1(\overline{T_1})- \lambda_1(\overline{T})\geq &~ X^t\big(D(\overline{T_1})-D(\overline{T})\big)X\\
		\geq &~ X^t\big(2A(T_1)-2A(T)\big)X\\
		=&~ 4(x_u-x_v)\sum_{z\in S}x_z\\
		\geq &~ 0.
	\end{align*}
	
	If $\lambda_1(\overline{T_1})= \lambda_1(\overline{T})$, then for the vertex $u$, we have $0=\lambda_1(\overline{T_1})x_u- \lambda_1(\overline{T})x_u\geq \sum_{z\in S}x_z>0$, a contradiction. Therefore, $\lambda_1(\overline{T_1})> \lambda_1(\overline{T})$. This completes the proof.
\end{proof}

\begin{thm}\label{distsign-comple}
	Let $T$ be a tree of order $n$ with $\diam(T)\geq 4$, and let $T_1$ be the tree obtained from $T$ by a proper GTS. Then $\mu_1(\overline{T_1})> \mu_1(\overline{T})$.
\end{thm}
\begin{proof}
	Let $z$ be the neighbour of $u$ lying on the $uPv$ path, and let $S_1=N_T(u)\setminus \{z\}$ and $S_2=N_T(v)\setminus \{w\}$. Suppose that $X=(x_1,x_2,\hdots,x_n)^t$ be the unit perron vector of $\overline{T}$ corresponding to $\mu_1(\overline{T})$. Without loss of generality, assume that $x_u\geq x_v$. Therefore, 
	\begin{align*}
		T_1=T-\{vv_2: v_2\in S_2\}+\{uv_2: v_2\in S_2\}.
	\end{align*} 
	
	If $v_1\in S_1$ and $v_2\in S_2$, then we have 
	\begin{align*}
		d_{\overline{T}}(u,v_1)=2, \qquad  d_{\overline{T_1}}(u,v_1)\geq 2;\\
		d_{\overline{T}}(u,v_2)=1, \qquad  d_{\overline{T_1}}(u,v_2)\geq 2;\\
		d_{\overline{T}}(v,v_2)=2, \qquad  d_{\overline{T_1}}(v,v_2)=1;
	\end{align*}
	and the distances of other pair of vertices in $\overline{T}$ and $\overline{T_1}$ are same. Now, by Rayley's principle, we have 
	\begin{align*}
		\mu_1(\overline{T_1})- \mu_1(\overline{T})\geq &~ X^t\big(D^Q(\overline{T_1})-D^Q(\overline{T})\big)X\\
		= &~2x_u \sum_{v_1\in S_1}\big( d_{\overline{T_1}}(u,v_1)- d_{\overline{T}}(u,v_1)\big)x_{v_1}+2x_u\sum_{v_2\in S_2}\big( d_{\overline{T_1}}(u,v_2)- d_{\overline{T}}(u,v_2)\big)x_{v_2}\\
		&~ +2x_v\sum_{v_2\in S_2}\big( d_{\overline{T_1}}(v,v_2)- d_{\overline{T}}(v,v_2)\big)x_{v_2}+x_u^2 \sum_{v_1\in S_1}\big( d_{\overline{T_1}}(u,v_1)- d_{\overline{T}}(u,v_1)\big)\\
		&~ +x_u^2\sum_{v_2\in S_2}\big( d_{\overline{T_1}}(u,v_2)- d_{\overline{T}}(u,v_2)\big)+x_v^2 \sum_{v_2\in S_2}\big( d_{\overline{T_1}}(v,v_2)- d_{\overline{T}}(v,v_2)\big)\\
		\geq &~ 2x_u \sum_{v_1\in S_1}\big( d_{\overline{T_1}}(u,v_1)- d_{\overline{T}}(u,v_1)\big)x_{v_1}+x_u^2 \sum_{v_1\in S_1}\big( d_{\overline{T_1}}(u,v_1)- d_{\overline{T}}(u,v_1)\big)\\
		&~ +2(x_u-x_v)\sum_{v_2\in S_2}x_{v_2}+(x_u^2-x_v^2)|S_2|\\
		\geq &~ (x_u-x_v)\Big(2\sum_{v_2\in S_2}x_{v_2}+(x_u+x_v)|S_2|\Big)\\
		\geq &~ 0.
	\end{align*}
	
	If $\mu_1(\overline{T_1})= \mu_1(\overline{T})$, then for the vertex $u$, we have $0=\mu_1(\overline{T_1})x_u- \mu_1(\overline{T})x_u\geq \sum_{v_2\in S_2}x_{v_2}+x_u |S_2|>0$, a contradiction. Therefore, $\mu_1(\overline{T_1})> \mu_1(\overline{T})$. This completes the proof.
\end{proof}

\begin{thm}\label{dalpha-comple}
	Let $T$ be a tree of order $n$ with $\diam(T)\geq 4$, and let $T_1$ be the tree obtained from $T$ by a proper GTS. Then $\rho_\alpha(\overline{T_1})> \rho_\alpha(\overline{T})$.
\end{thm}
\begin{proof}
	Let $z$ be the neighbour of $u$ lying on the $uPv$ path, and let $S_1=N_T(u)\setminus \{z\}$ and $S_2=N_T(v)\setminus \{w\}$. Suppose that $X=(x_1,x_2,\hdots,x_n)^t$ be the unit perron vector of $\overline{T}$ corresponding to $\rho_\alpha(\overline{T})$. Without loss of generality, assume that $x_u\geq x_v$. Therefore, 
	\begin{align*}
		T_1=T-\{vv_2: v_2\in S_2\}+\{uv_2: v_2\in S_2\}.
	\end{align*} 
	
	If $v_1\in S_1$ and $v_2\in S_2$, then we have 
	\begin{align*}
		d_{\overline{T}}(u,v_1)=2, \qquad  d_{\overline{T_1}}(u,v_1)\geq 2;\\
		d_{\overline{T}}(u,v_2)=1, \qquad  d_{\overline{T_1}}(u,v_2)\geq 2;\\
		d_{\overline{T}}(v,v_2)=2, \qquad  d_{\overline{T_1}}(v,v_2)=1;
	\end{align*}
	and the distances of other pair of vertices in $\overline{T}$ and $\overline{T_1}$ are same. Now, by Rayley's principle, we have 
	\begin{align*}
		\rho_\alpha(\overline{T_1})- \rho_\alpha(\overline{T})\geq &~ X^t\big(D-\alpha(\overline{T_1})-D_\alpha(\overline{T})\big)X\\
		= &~2(1-\alpha) x_u \sum_{v_1\in S_1}\big( d_{\overline{T_1}}(u,v_1)- d_{\overline{T}}(u,v_1)\big)x_{v_1}+2(1-\alpha)x_u\sum_{v_2\in S_2}\big( d_{\overline{T_1}}(u,v_2)- d_{\overline{T}}(u,v_2)\big)x_{v_2}\\
		&~ +2(1-\alpha)x_v\sum_{v_2\in S_2}\big( d_{\overline{T_1}}(v,v_2)- d_{\overline{T}}(v,v_2)\big)x_{v_2}+\alpha x_u^2 \sum_{v_1\in S_1}\big( d_{\overline{T_1}}(u,v_1)- d_{\overline{T}}(u,v_1)\big)\\
		&~ +\alpha x_u^2\sum_{v_2\in S_2}\big( d_{\overline{T_1}}(u,v_2)- d_{\overline{T}}(u,v_2)\big)+\alpha x_v^2 \sum_{v_2\in S_2}\big( d_{\overline{T_1}}(v,v_2)- d_{\overline{T}}(v,v_2)\big)\\
		\geq &~ 2(1-\alpha)x_u \sum_{v_1\in S_1}\big( d_{\overline{T_1}}(u,v_1)- d_{\overline{T}}(u,v_1)\big)x_{v_1}+\alpha x_u^2 \sum_{v_1\in S_1}\big( d_{\overline{T_1}}(u,v_1)- d_{\overline{T}}(u,v_1)\big)\\
		&~ +2(1-\alpha)(x_u-x_v)\sum_{v_2\in S_2}x_{v_2}+\alpha(x_u^2-x_v^2)|S_2|\\
		\geq &~ (x_u-x_v)\Big(2(1-\alpha)\sum_{v_2\in S_2}x_{v_2}+\alpha(x_u+x_v)|S_2|\Big)\\
		\geq &~ 0.
	\end{align*}
	
	If $\rho_\alpha(\overline{T_1})= \rho_\alpha(\overline{T})$, then for the vertex $u$, we have $0=\rho_\alpha(\overline{T_1})x_u- \rho_\alpha(\overline{T})x_u\geq (1-\alpha)\sum_{v_2\in S_2}x_{v_2}+\alpha x_u |S_2|>0$, a contradiction. Therefore, $\rho_\alpha(\overline{T_1})> \rho_\alpha(\overline{T})$. This completes the proof.
\end{proof}

Now, using the above theorems, we determine the unique tree whose
complement has minimum distance spectral radius, distance signless Laplacian spectral radius, and $D_\alpha$-spectral
radius, respectively.

\begin{thm}
	Let $T$ be a tree of order $n\geq 4$ and $T\ncong K_{1,n-1},P_n$. Then 
	\begin{enumerate}
		\item $\lambda_1(\overline{T})>\lambda_1(\overline{P_n})$.
		\item $\mu_1(\overline{T})>\mu_1(\overline{P_n})$.
		\item $\rho_{\alpha}(\overline{T})>\rho_{\alpha}(\overline{P_n})$.
	\end{enumerate}
\end{thm}
\begin{proof}
	It follows from Theorem \ref{GTS-minimal} that any tree $T$ of order $n$ can be obtained from $P_n$ by repeatedly applying the graft transformation mentioned in Definition \ref{GTS}. Now, the proof of $1$, $2$ and $3$ follows from Theorem \ref{dist-comple}, Theorem \ref{distsign-comple} and Theorem \ref{dalpha-comple}, respectively. 
\end{proof}

\bibliographystyle{ieeetr}
\bibliography{ref}
\end{document}